\documentclass[a4paper, 11pt]{amsart}

\usepackage{enumerate}
\usepackage{mathrsfs}
\usepackage{amssymb}
\usepackage[all]{xy}

\title[]{Okounkov bodies for ample line bundles with applications to 
multiplicities for group representations}
\author{Henrik Sepp\"anen*} 

\thanks{*supported by the DFG Priority Programme 1388 ``Representation 
Theory"}

\date{\today}

\address{Henrik Sepp\"{a}nen,
Mathematisches Institut,
Georg-August-Universit\"at G\"ottingen,
Bunsenstra\ss e 3-5, 
D-37073 G\"ottingen,
Germany}
\email{hseppaen@uni-math.gwdg.de}

\newcommand{\C}{\mathbb{C}}
\newcommand{\R}{\mathbb{R}}
\newcommand{\N}{\mathbb{N}}
\newcommand{\Z}{\mathbb{Z}}

\newtheorem{prop}{Proposition}[section]
\newtheorem{lemma}[prop]{Lemma}

\newtheorem{thm}[prop]{Theorem}
\theoremstyle{definition}

\newtheorem{rem}[prop]{Remark}
\newtheorem{example}{Example}

\begin{document}
 
\maketitle

\begin{abstract}
Let $\mathscr{L} \rightarrow X$ be an ample line 
bundle over a complex normal projective variety $X$. 
We construct a flag 
$X_0 \subseteq X_1 \subseteq \cdots \subseteq X_n=X$ of 
subvarieties for which the associated Okounkov body for 
$\mathscr{L}$ is a rational polytope. In the case when $X$ is 
a homogeneous surface, and the pseudoeffective cone of $X$ is rational 
polyhedral, we also show that the global Okounkov body is a rational polyhedral 
cone if the flag of subvarieties is suitably chosen.
Finally, we provide an application to the asymptotic study of group 
representations.
\end{abstract}

\section{Introduction}
The \emph{Okounkov body} of a line bundle 
$\mathscr{L} \rightarrow X$ over an $n$-dimensional complex 
variety is a convex set in $\R^n$ which carries information about the
section ring $\bigoplus_{k=0}^\infty H^0(X,\mathscr{L}^k)$ of $\mathscr{L}$. 
The idea is to construct a valuation-like function (meaning that 
it has the properties of a valuation in the ring-theoretic sense,
even if it only defined on homogeneous elements)
$$v: \bigsqcup_{k \in \N} H^0(X,\mathscr{L}^k)\setminus \{0\}  \rightarrow \N_0^n,$$ 
and define the semigroup 
\begin{equation*}
S(\mathscr{L},v):=\{(k,v(s)) \mid s \in H^0(X,\mathscr{L}^k) \setminus \{0\}\}. 
\end{equation*}
The associated Okounkov body is then defined as 
\begin{equation}
\Delta(\mathscr{L},v):=
\overline{\mbox{conv}}\left\{\frac{1}{k} (k,v(s)) \mid  s \in H^0(X,\mathscr{L}^k) 
\setminus \{0\}\right\}.
\end{equation}
Roughly speaking, the function $v$ will be defined as the successive orders 
of vanishing of sections along a flag 
$X_0 \subseteq X_{1} \subseteq \cdots \subseteq X_{n-1} \subseteq X_n=X$ 
of irreducible nonsingular subvarieties such that $X_i$ has dimension $i$.

These bodies were introduced by Okounkov in \cite{Ok96} from a 
representation-theoretic point of view. However, he only considered the 
semigroup defined by the values of $U$-invariants, where $U$ is the unipotent
radical of a Borel subgroup, $B$, of $G$. It turned out that
the elements of height $k$ in this semigroup carry information about 
the decomposition under $G$ of the $k$th piece, $ H^0(X,\mathscr{L}^k)$, and 
that the Euclidean volumes 
of certain slices of the associated convex body describe the
asymptotics of the decomposition.

It has since then become an interesting problem \emph{per se} to study the 
semigroups and associated convex bodies defined by considering the values 
of \emph{all} sections of a line bundle, without the assumption of
a group action (\cite{LM09}), and their relation to the 
geometry of the line bundle $\mathscr{L}$.
One crucial connection between the section ring of $\mathscr{L}$ and the
Okounkov body $\Delta(\mathscr{L},v)$ is the identity
\begin{equation}
\mbox{vol}\Delta(\mathscr{L},v)=\lim_{k \rightarrow \infty}
\frac{\mbox{dim} H^0(X,\mathscr{L}^k)}{k^n},
\end{equation}
which holds for the linear series of a big line bundle over
an $n$-dimensional complex projective variety $X$ (see \cite{LM09}).

A fundamental question to ask is whether $\Delta(\mathscr{L},v)$ is 
a rational polytope, i.e., the convex hull of finitely many 
points in $\mathbb{Q}^n$. This was conjectured by Okounkov
in \cite{Ok96} in his setting where he only considered $U$-invariants. 
It is true in some cases, e.g., equivariant line 
bundles over toric varieties. This follows from recent results 
by Kaveh and Khovanskii where they consider a setting that
generalizes Okounkov's, namely group representations on 
graded algebras of meromorphic functions on $X$, cf. \cite{KK10}.
In fact, the Okounkov body of a torus-equivariant line
bundle over a projective toric variety equals the \emph{moment polytope} 
(cf. \cite{Br86}) associated to the graded representation.
Kaveh (\cite{Kav11}) has studied this problem for a 
line bundle over a full flag variety $G/B$. He considers an 
Okounkov body defined by a valuation determined by 
a local system of coordinates on a Bott-Samelson resolution, and shows 
that this Okounkov body is a polytope by identifying it with a 
Littelmann string polyope.

It is, however, not true in general that $\Delta(L,v)$ is 
a rational polytope. A counterexample can be found in 
\cite[Section 6.3]{LM09}.\\

In this paper, we prove that if $\mathscr{L}$ is ample, there
exists a flag of nonsingular irreducible subvarieties
$X_0 \subseteq X_1 \subseteq \cdots \subseteq X_n:=X$ for which the 
associated Okounkov body is a rational polytope. 
In the case when $X$ is a homogenous surface, we can say a little more: If 
the pseudoeffective cone $\overline{\mbox{Eff}}(X)$ is a rational 
polyhedral cone, and $D$ is a fixed ample divisor on $D$, then 
for a very general choice of flag $X_0 \subseteq X_1$, where 
$X_1$ is an irreducible divisor in the linear system $|D|$, the 
global Okonkov body of $X$ with respect to this flag is a rational 
polyhedral cone. 

We would like to point out that the first result, Theorem 3.3, which deals 
with a given ample line bundle, has  also been obtained independently - but 
with a different proof - by Anderson, K\"uronya and Lozovanu in \cite{akl}. 
In fact, they even prove the result under the weaker assumption of having a semi-ample 
line bundle. We like to present an alternative proof for this known result nevertheless, 
since the proof can be adapted to the global situation for homogeneous surfaces in 
Section 4. Moreover, we are working on pushing these techniques a little further in a more general 
context in a current joint project with D. Schmitz (\cite{SS14}). 

In the final section we address the original motivation 
for Okounkov, namely that of studying multiplicities for group representations, 
by applying the result to certain GIT-quotients; thus providing a 
``polyhedral" expression for asymptotic branching laws, which is by now a well 
established goal in representation theory.

\noindent {\bf Acknowledgement}. 
The author would like to thank Jacopo Gandini and Joachim Hilgert
for interesting discussions on Okounkov bodies, as well as Shin-Yao Jow and 
David Schmitz for helpful comments on an earlier version of this manuscript.

\section{Flags of normal subvarieties}
Let $X$ be a normal complex variety with singular locus $\mbox{sing}(X)$ 
and set of regular points $X^{\mbox{reg}}$, and let 
$D \subseteq X$ be a prime Cartier divisor given as the zero set of the 
section $\xi \in H^0(X, \mathcal{O}_X(D))$. If 
$Y \subseteq X$ is any effective Cartier divisor of $X$ 
and $s \in H^0(X, \mathcal{O}_X(Y))$ is a section of 
$\mathcal{O}_X(Y)$ that vanishes to order $a \in \N_0$ along 
$D$, the quotient $s/\xi^a$ defines a meromorphic section 
of $\mathcal{O}_X(D-aY)$. In fact, this is even a regular 
section. Indeed, if $\{U_\alpha\}$ is a finite open 
cover of $X$ such that $\mathcal{O}_X(D)$ is trivial 
over each $U_\alpha$ and the section $\xi$ 
is represented by the regular function 
$f_\alpha \in \mathcal{O}_X(U_\alpha)$ on $U_\alpha$ 
via this local trivialization, then the section 
$s/\xi^a$ is regular on the set
\begin{align*}
V_\alpha:= \{x \in D \cap U_\alpha^{\mbox{reg}} \mid df_\alpha(x) \neq 0\},
\end{align*}
where $U_\alpha^{\mbox{reg}}:=U_\alpha \cap X^{\mbox{reg}}$,
since $f_\alpha$ vanishes to order one along $D$ on this set, and 
any local trivialization of $s$ vanishes to order $a$.
Since $s/\xi^a$ is clearly regular outside $D$, it 
follows that it is regular outside the closed set
\begin{align*}
Z:=\bigcup_\alpha D \cap \overline{\{ x \in U^{\mbox{reg}}_\alpha \mid df_\alpha(x)=0\}}
\cup (D \cap \mbox{sing}(X)).
\end{align*}
Since $\xi$ vanishes to order one along $D$, each subset
$D \cap \overline{\{ x \in U^{\mbox{reg}}_\alpha \mid df_\alpha(x)=0\}}$ 
is a proper closed subset of $D$, and hence of codimension at least two 
in $X$. Moreover, since $X$ is normal, the singular locus $\mbox{sing}(X)$
is also of codimension at least two. Hence, $Z$ is of codimension 
at least two in $X$. Since $s/\xi^a$ is regular outside $Z$, it thus follows 
from the normality of $X$ that $s/\xi^a$ is regular everywhere.

{\bf Flag of normal varieties, construction of Okounkov body}.
Assume now that we have a flag 
$X_0 \subseteq X_1 \subseteq X_{n-1} \subseteq X_n=X$
of irreducible normal subvarieties with $\mbox{dim} X_i=i$ for $i=0,\ldots, n$
and that for each $i=2,\ldots, n$, the subvariety $X_{i-1}$ is 
a Cartier divisor of $X_i$ given as the zero set of a section 
$\xi_{n-i+1} \in \mathcal{O}(X_{i-1}$. Let $\mathscr{L}_{n-i+1}$ 
be the corresponding line bundle over $X_i$.
If $\mathscr{L}$ is an effective Cartier divisor of $X$, we
define a valuation-like function 
\begin{align*}
v: \bigsqcup_{k=0}^\infty H^0(X, \mathscr{L}^k) \setminus \{0\}
\rightarrow \N_0^n.
\end{align*}
inductively as follows.

If $s \in H^0(X, \mathcal{L}^k)$ vanishes to order $a_n$ along $X_{n-1}$, 
then, by the above,  $s/\xi_1^{a_n}$ is a regular section of 
$H^0(X, \mathscr{L}^k \otimes \mathscr{L}_1^{a_n})$ 
which does not vanish identically on $X_{n-1}$. 
Assume now that $a_n,\ldots, a_{n-i+1}$ are defined in such a way 
that $s/(\xi_1^{a_1} \cdots \xi_i^{a_{n-i+1}})$ defines a regular 
section of 
$\mathscr{L}^k \otimes \mathscr{L}_1^{-a_1} 
\otimes \cdots \otimes \mathscr{L}_{i}^{-a_{n-i+1}}$ over $X_{n-i+1}$
which does not vanish identically on $X_{n-i}$. Let 
$a_{n-i}$ be the order of vanishing along $X_{n-i-1}$ of 
the restriction to $X_{n-i}$ of   $s/(\xi_1^{-a_1} \cdots \xi_i^{-a_{n-i+1}})$. 
By the first paragraph of this section, $s/(\xi_1^{-a_1} \cdots \xi_{i+1}^{-a_{n-i}})$
is a regular section of $\mathscr{L}^k \otimes \mathscr{L}_1^{-a_1} 
\otimes \cdots \otimes \mathscr{L}_{i+1}^{-a_{n-i}}$ over $X_{n-i}$
which does vanish identically on $X_{n-i-1}$.

We then define 
\begin{align*}
v(s):=(a_1,\ldots, a_n) \in \N_0^n.
\end{align*}
Then $v$ is an $\N_0^n$-valued \emph{valuation-like function} on 
$\bigsqcup_{k=0}^\infty H^0(X, \mathscr{L}^k) \setminus \{0\}$
with respect to the inverse lexicographic order on $\N_0^n$.
We now use $v$ to define the semigroup
\begin{align*}
S(\mathscr{L}, v):=\{(k, v(s)) \mid s \in H^0(X, \mathscr{L}^k \setminus \{0\})\}
\subseteq \N_0 \times \N_0^n
\end{align*}
and the closed convex cone
\begin{align*}
C(\mathscr{L}, v) \subseteq \R \oplus \R^n
\end{align*}
generated by the semigroup $S(\mathscr{L}, v)$. 
Finally, the Okounkov body is
\begin{align*}
\Delta(\mathscr{L}, v)=\Delta_{X_\bullet}(\mathscr{L}):=
C(\mathscr{L}, v) \cap (\{1\} \times \R^n).
\end{align*}

\begin{rem}
The construction can also be carried out by using 
a local trivialization of $\mathscr{L}$ to embed the 
section ring $R(L):=\bigoplus_{k=0}^\infty H^0(X, \mathscr{L}^k)$
into the ring of rational functions $\C(X)$ and to 
use the flag to define a proper $\Z^n$-valued valuation on the 
ring $\C(X)$. For this, the assumption that $X_{i-1}$ 
be a Cartier divisor of $X_i$ is superfluous. In fact, by 
passing to normalizations, even the condition of normality of the $X_i$ 
can be omitted for the construction (cf. \cite{Ok96}).
We are, however, interested in having an $\N_0^n$-valued 
valuation-like function and not merely a $\Z^n$-valued one.
Note in particular, that if $v(s)=(a_1,\ldots, a_n)$, then, by the 
construction, the line bundle $\mathscr{L}^k \otimes \mathscr{L}_1^{-a_1} 
\otimes \cdots \otimes \mathscr{L}_{i}^{-a_{n-i+1}}$ over $X_{n-i+1}$ 
 is effective for each $i \in \{1,\ldots, n\}$. This property will be crucial in the 
proof of the main result.
\end{rem}

\section{A polyhedral Okounkov body}
Let $X$ be a projective normal $n$-dimensional complex variety and 
let $\mathscr{L} \rightarrow X$ be an ample line bundle. 
Since Okounkov bodies satisfy the property $\Delta(\mathscr{L}^k)=k\Delta(\mathscr{L})$, 
$k \in \N$ (cf. \cite[Proposition 4.1.]{LM09}), we may without loss of
generality assume that $\mathscr{L}$ is very ample.

By a Bertini-type theorem by Seidenberg (cf. \cite{S50}), for 
a generic tuple $(\xi_1,\ldots, \xi_{n-1})$ of
sections $\xi_1,\ldots, \xi_{n-1} \in H^0(X, \mathscr{L})$ the 
zero sets
\begin{align}
X_i:=\{x \in X \mid \xi_1(x)=\cdots =\xi_{n-i}(x)=0\}, \quad i=1,\ldots, n-1
\end{align}
define irreducible normal subvarieties of $X$ with 
$\mbox{dim} X_i=i$, $i=1,\ldots, {n-1}$. Now, fix such a tuple $(\xi_1,\ldots, \xi_{n-1})$
and let 
$X_0:=\{p\} \subseteq X_1$ for some point $p \in X_1$.
Let 
\begin{align*}
v: \bigsqcup_{k \geq 0} H^0(X, \mathscr{L}^k) \setminus \{0\}
\rightarrow \N^n
\end{align*}
be the valuation-like function defined by 
the flag 
\begin{align}
X_0 \subseteq X_1\subseteq \cdots \subseteq X_{n-1} \subseteq X_n:=X,
\label{E: flag}
\end{align}
and let $\Delta \subseteq \R^n$ be the 
associated Okounkov body. 
Similarly, let 
\begin{align*}
v_1: \bigsqcup_{k \geq 0} H^0(X_1, \mathscr{L}^k\mid_{X_1}) \setminus \{0\}
\rightarrow \N
\end{align*}
be the valuation-like function defined by the flag $X_0 \subseteq X_1$,
and let $\Delta_1 \subseteq \R$ be the corresponding Okounkov body.
Then 
\begin{align*}
\Delta_1=[0, b] \subseteq \R,
\end{align*} 
where $b=\mbox{deg} \,\mathscr{L}\mid_{X_1}$ (cf. \cite[Ex. 1.14]{LM09}).
\begin{lemma} \label{L: okincl}
The inclusion 
\begin{align*}
[0,b] \times \{0\}\times \cdots \times \{0\} \subseteq \Delta
\end{align*}
holds.
\end{lemma}

\begin{proof}
By the closedness of $\Delta$ it suffices to prove that 
$(c,0,\ldots, 0) \in \Delta$ for every $0<c<b$. 

By the ampleness of 
$\mathscr{L}$ there exists an $N_0 \in \N$ such that 
the restriction maps
\begin{align*}
R_j: H^0(X, \mathscr{L}^j) \rightarrow H^0(X_1, \mathscr{L}^j\mid_{X_1}) 
\end{align*}
are surjective for $j \geq N_0$. 
If now $c \in (0,b)$, choose a point $v_1(\tau)/m \in (c, b)$ for a section 
$\tau \in H^0(X_1, \mathscr{L}^m\mid_{X_1})$.
Hence, for sufficiently large $N \in \N$,
there exists a section
$\widetilde{\tau_N} \in H^0(X, \mathscr{L}^{Nm})$
with $R_{Nm}(\widetilde{\tau_N})=\tau^N$. 
Then $v(\widetilde{\tau_N})=(v_1(\tau^N),0,\ldots, 0)=(Nv_1(\tau),0,\ldots, 0)$, 
so that $(v_1(\tau)/m,0,\ldots, 0) \in \Delta$. By the convexity of 
$\Delta$ we then have $(c,0,\ldots, 0) \in \Delta$.
\end{proof}

\begin{rem}
The lemma would of course also hold if the subvarieties $X_i$ 
were defined by unequal line bundles $\mathscr{L}_i$, $i=1,\ldots, n-1$. 
For a related result for line bundles $\mathscr{L}$ that are not ample, 
but merely big, cf. \cite[Theorem B]{jow}.
\end{rem}

Let $e_1,\ldots, e_n$ be the standard basis for $\R^n$.
\begin{thm} \label{T: okamp}
The Okounkov body $\Delta$ is the
convex hull of the set 
$$\{0, be_1, e_2, \ldots, e_n\}.$$
\end{thm}

\begin{proof}
First of all, $v(\xi_{n-i+1})=e_i$ for $i=2,\ldots, n$. 
Since the linear system defined by $\mathscr{L}$ is basepoint-free, 
$0 \in v(H^0(X, \mathscr{L}) \setminus \{0\})$. By Lemma \ref{L: okincl}
$be_1 \in \Delta$. Hence the convex hull of the points
$0,be_1, e_2,\ldots, e_n$ is a subset of $\Delta$.

It thus suffices to prove that for every $a=v(s)$ for some 
$s \in H^0(X, \mathscr{L}^k) \setminus \{0\}$ there exist 
$x_0,\ldots, x_{n} \geq 0$ 
such that 
\begin{align}
&x_0+\cdots+x_{n}=k, \nonumber\\ 
&a=x_0\cdot 0+x_1 be_1+x_2e_2+\cdots +x_{n}e_n. 
\label{E: convcomb}
\end{align}
For this, let $a=(a_1,\ldots, a_m,0,\ldots, 0)=v(s)$ for 
$s \in H^0(X, \mathscr{L}^k)$ and assume that $a_m\geq 1$. 
Then $(s/\xi_{n-m+1}^{a_m})\mid_{X_{m-1}}$ defines a section in \\
$H^0(X_{m-1}, \mathscr{L}^{k-a_m}\mid_{X_{m-1}})$, and by iteration 
we can write 
\begin{align*}
a=(a_1,0,\ldots, 0)+\sum_{i=2}^m a_iv(\xi_{n-i+1}),
\end{align*} 
where 
$a_1=v_1(t)$ for the section 
$$t=(s/(\xi_{n-1}^{a_2}\cdots \xi_{n-m+1}^{a_m}))\mid_{X_1} \in 
H^0(X_1, \mathscr{L}^{k-\sum_{i=2}^ma_i}\mid_{X_1}).$$
Since the line bundle $\mathscr{L}^{k-\sum_{i=2}^ma_i}\mid_{X_1}$ is 
effective we must have $p:=k-\sum_{i=2}^ma_i \geq 0$. Indeed, 
$\mathscr{L}^r\mid_{X_1}$ is effective for every $r \geq 0$, 
and hence $\mathscr{L}^r \mid_{X_1}$ cannot be effective for 
any $r<0$. 
Since $\Delta_1=[0,b]$ there exist real $x_0, x_1 \geq 0$ 
with $x_0+x_1=p$ and 
\begin{align}
a_1=x_0 \cdot 0+x_1 b. \label{E: conv1},
\end{align}
and hence
\begin{align*}
a_1e_1=x_0 \cdot 0+x_1be_1.
\end{align*}
By putting 
\begin{align*}
x_i:=a_i \quad i=2,\ldots, n,
\end{align*}
we have thus found nonnegative real numbers $x_0,\ldots, x_n$ satisfying 
\eqref{E: convcomb}.
\end{proof}

\section{Global Okounkov bodies for homogeneous surfaces}
In this section we prove that, if $X$ is a homogeneous projective surface admitting 
a rationally polyhedral pseudoeffective cone, the global Okounkov 
body, with respect to a suitably chosen flag of subvarieties, is a rational polyhedral 
cone. \\

Let $X$ be a normal projective complex variety. Let $N^1(X)_\R \cong \R^r$, 
where $r$ is the Picard number of $X$,  be its real N\'eron-Severi space, and 
 let $\overline{\mbox{Eff}}(X) \subseteq N^1(X)_\R$ be the pseudoeffective cone
of $X$. 

Let $X_0 \subseteq \cdots \subseteq X_n=X$ be a flag of irreducible subvarieties as in 
the previous sections.
We recall that the global Okounkov body, $\Delta_{X_\bullet}(X)$, of $X$ with respect to the 
above flag is the closed convex cone in $\R^n \oplus N^1(X)_\R$ 
generated by the semigroup
\begin{align*}
S(X, v)&:=\{(v(s), [L]) \mid L \quad \mbox{big line bundle},\,\, s \in H^0(X, L) \setminus \{0\}\}\\
&\subseteq \N_0^n \oplus N^1(X)_\Z.
\end{align*}
(cf. \cite{LM09}).
Here there valuation-like function is extended to the (nonzero homogeneous elements of) 
ring generated by the spaces of sections  of all effective line bundles on $X$.
\\

In order to adapt the idea of the proof of the main theorem in the 
previous section to a global situtation, we first need the following lemma.

\begin{lemma} \label{L: joweff}
Let $X$ be a complex normal projective variety, and 
let $X_\bullet$ be an admissible flag of subvarieties.
If $L$ is a nef line bundle on $X$, the 
inclusion 
\begin{align*}
[0, (L \cdot X_1)]\times \{0\}^{n-1} \subseteq \Delta_{X_\bullet}(L)
\end{align*}
holds.
\end{lemma}

\begin{proof}
First of all, we note that, by Serre vanishing, for any ample line bundle 
$E$, the restriction map
\begin{align*}
R^k: H^0(X, E^k) \rightarrow H^0(X_1, E^k\mid_{X_1})
\end{align*}
is surjective for sufficienly big $k \in \N$. 
 
To simplify the notation, we will now work with the corresponding 
Weil divisor. Hence, let $F$ be a Weil divisor, such that 
$\mathcal{O}_X(F)=L$. 


It now suffices to prove that 
$(a/j,0,\ldots, 0) \in \Delta_{X_\bullet}(F)$, for  
$$a \in v_1(H^0(X_1, \mathcal{O}_X(jF)\mid_{X_1})), \quad j \in \N_0,$$
where $v_1$
is the valuation-like function, defined  on (the nonzero homogenous elements of) the 
ring of all sections of all effective line bundles on $X_1$,  with respect to  the point 
$X_0 \in X_1$.
For such an $a$, we have
\begin{align*}
ma \in v_1(H^0(X_1, \mathcal{O}_X(mjF+D)\mid_{X_1})), \quad m \in \N, 
\end{align*}
since $0 \in v_1(H^0(X_1, \mathcal{O}_X(D)\mid_{X_1}))$ because $D$ 
is ample. Since $mjF+D$ is ample, $(ma, 0,\ldots, 0) \in \Delta_{X_\bullet}(mjF+D)$
by Lemma \ref{L: okincl}. Hence, 
\begin{align*}
(a,0,\ldots, 0) \in \Delta_{X_\bullet}\left(jF+\frac{D}{m}\right), \quad m \in \N, 
\end{align*}
In particular, $((a,0), jF+\frac{D}{m}) \in \Delta_{X_\bullet}(X)$, for 
every $m \in \N$. Since $\Delta_{X_\bullet}(X)$ is closed, we thus have 
$\lim_{m \rightarrow \infty} ((a,0), jF+\frac{D}{m})=((a,0), jF) \in \Delta_{X_\bullet}(X)$ 
(cf. also \cite[Lemma 8]{akl}). Hence, $(a, 0,\ldots, 0) \in \Delta_{X_\bullet}(jF)$, 
so that
\begin{align*}
(a/j, 0,\ldots, 0) \in \Delta_{X_\bullet}(F).
\end{align*}
\end{proof}

\begin{rem}
If $L$ is basepoint-free, the conclusion of Lemma \ref{L: joweff}, for a flag $X_\bullet$ 
defined  by a \emph{very general} choice of ample divisors $(D_1,\ldots, D_n)$,  
also follows from \cite[Theorem B]{jow}. Indeed, if $L$ is basepoint-free, the restricted 
volume on the left hand side of \cite[Theorem B]{jow} is given by the intersection 
number $(L \cdot X_1)$ (cf. \cite[Cor. 3.3.]{jow}).
\end{rem}

\noindent We are now ready to state the main result of this section.
Recall that a variety $X$ is called \emph{homogeneous} if $X$ carries a 
transitive action of a connected algebraic group $G$.

\begin{thm} \label{T: globalbody} Let $X$ be a homogenous projective 
surface with a rational polyhedral pseudoeffective cone $\overline{\mbox{Eff}}(X)$, 
and let $\xi_1,\ldots, \xi_r$ be integral generators 
of $\overline{\mbox{Eff}}(X)$. Let 
$D$ be a very ample divisor on $X$. For a generic 
choice of flag $X_0 \subseteq X_1$, where $X_1 \in |D|$, the 
global Okounkov body $\Delta_{X_\bullet}(X)$ is the rational 
polyhedral cone generated by the 
vectors 
\begin{align*}
((0,0), \xi_1), ((d(\xi_1),0), \xi_1),\ldots, ((0,0), \xi_r), ((d(\xi_r),0), \xi_r), 
((0,1), [D]),
\end{align*}
where $d(\xi_j)$ is the degree of $\xi_j$ along the curve $X_1$.\\
\end{thm}

\begin{proof}
First of all, every effective line bunde on $X$ is nef (cf. \cite[Example 1.4.7.]{laz}), 
and, hence, Lemma \ref{L: joweff} holds for every effective line bundle on $X$.
 
Moreover, there is a countable family $\{L_j\}_{j \in \N}$ of line 
bundles on $X$ 
such that each line 
bundle $L$ is numerically equivalent to some $L_j$. In particular, for 
any flag $X_\bullet$ of subvarieties, the global Okounkov 
body $\Delta_{X_\bullet}(X)$ can be defined using only the 
line bundles $L_j, \,j \in \N$. We can without loss of generality assume that 
the $L_j$ form a lattice, and that this lattice contains all divisors $kD$, for $k \in D$.
Now, for a generic tuple $(X_0, D_1,\ldots, D_{n-1}) \in X \times |D|^{n-1}$, with $X_0 \in 
D_{1} \cap \cdots \cap D_{n-1}$, 
Lemma \ref{L: joweff} holds for each effective $L_j$.
Under these assumptions we now have 
\begin{align*}
[0, d(\xi_j)] \times \{0\} \subseteq \Delta_{X_\bullet}(\xi_j) \subseteq \Delta_{X_\bullet}(X), 
\quad j=1,\ldots, r.
\end{align*}
Hence, the global Okounkov body $\Delta_{X_\bullet}(X)$ contains all 
the vectors in the claim. \\

In order to prove that the given vectors generate the cone $\Delta_{X_\bullet}(X)$, 
let $(v(s), [L]) \in S(X, v)$, with $v(s)=(a_1, a_2)$. Then,
$E:=L \otimes \mathcal{O}_X(-a_2D)$ is effective. 
Hence, 
\begin{align*}
[E]=t_1\xi_1+\cdots+t_r \xi_r, 
\end{align*}
for some  $t_1,\ldots, t_r \geq 0$.
The Okounkov body, with respect to the point $X_0$,  of the restriction of 
$E$ to $X_1$ 
is given by 
\begin{align*}
\Delta_{X_0}(E\mid_{X_1})=[0,(E \cdot X_1)].
\end{align*}
Hence, $a_1 \in [0,(E \cdot X_1)]$, so that 
$a_1=c (E \cdot X_1)$, for some $c \in [0,1]$. Since 
\begin{align*}
(E \cdot X_1)=t_1 d(\xi)+\cdots+t_r d(\xi_r), 
\end{align*}
we have 
$$a_1=\sum_{i=1}^r c t_id(\xi_i).$$
It follows that 
\begin{align*}
((a_1, a_2), [E])=\sum_{i=1}^r t_i(1-c)((0,0), \xi_i)
+\sum_{i=1}^r ct_i((d(\xi_i), 0), \xi_i)+
a_2((0,1), [D]).
\end{align*}
This shows that the closed convex cone generated by the vectors in the claim 
contains the semigroup $S(X, v)$, and this proves the theorem.
\end{proof}

\begin{example}
If $X=\mathbb{P}^1 \times \mathbb{P}^1$ with projections $p_1, p_2$ onto the 
respective factors, the pseudoeffective cone is generated by 
the divisors $D_1$ and $D_2$, where $\mathcal{O}_X(D_i)=p_i^*\mathcal{O}(1), \, i=1,2$. 
The intersection relations are given by
\begin{align*}
D_1^2=D_2^2=0, \quad D_1 \cdot D_2=1.
\end{align*}
If we choose the ample divisor as $D:=D_1+D_2$, the global 
Okounkov body with respect to a generic flag is 
thus given as the closed cone generated by the vectors
\begin{align*}
((0,1), [D]), ((0,0), [D_1]), ((0,0), [D_2]), ((1,0), [D_1]), ((1,0), [D_2]).
\end{align*}

\end{example}

\begin{example}
If $D$ is the ample divisor of Theorem \ref{T: globalbody}, we know 
from Theorem \ref{T: okamp} that the Okounkov body $\Delta_{X_\bullet}(\mathcal{O}_X(D))$ of 
$D$ is the triangle generated by the vectors $(0,0), (D\cdot D, 0)$, and $(0,1)$. 
In order to see that $\Delta_{X_\bullet}(\mathcal{O}_X(D))=\Delta_{X_\bullet}(D)$ 
occurs as a slice of the global Okounkov body $\Delta_{X_\bullet}(X)$, it 
suffices to see that the vector $((D \cdot D, 0), [D])$ is a nonnegative 
linear combination of the generators in Theorem \ref{T: globalbody}. 
For this, write $[D]$ as 
$[D]=t_1\xi_1+\cdots+t_r\xi_r$, with $t_1,\ldots, t_r \geq 0$. 
Then $D \cdot D=\sum_{i=1}^r t_i \xi_i \cdot D=\sum_{i=1}^r t_i d(\xi_i)$, so 
that 
$((D\cdot D, 0), [D])=\sum_{i=1}^r t_i((d(\xi_i), 0), \xi_i)$. 
This shows that $\Delta_{X_\bullet}(D)$ is the preimage of 
$[D]$ under the map $\Delta_{X_\bullet}(X) \rightarrow \overline{\mbox{Eff}}(X), 
\quad (x,y) \mapsto y$.
\end{example}

\begin{rem}
The only surface which occurs as a flag variety $G/P$, for a complex 
reductive group $G$ and a parabolic subgroup $P$, is 
$\mathbb{P}^1 \times \mathbb{P}^1$.
On the other hand, Bauer gives a classification of the abelian surfaces 
admitting a rational polyhedral pseudoeffective cone in \cite{B}.
\end{rem}

\section{Multiplicities for group representations}
The purpose of this section is to apply the result from the previous sections 
to a setting occurring in the context of representations of 
complex reductive groups. 

Let $X$ be a normal projective variety, and let $\mathscr{L} \rightarrow X$
be an ample line bundle. Assume that $L$ is a complex reductive group 
which acts on $X$ and that this action lifts to an action on 
$\mathscr{L}$ by bundle automorphisms. Then there is 
a natural representation of $L$ on each space
$H^0(X, \mathscr{L}^k)$, for $k \in \N_0$.
We are then interested in finding the decomposing 
under $L$ of $ H^0(X,\mathscr{L}^k)$;
\begin{align*}
H^0(X, \mathscr{L}^k)=\bigoplus_{\mu}m(k,\mu)W_\mu.
\end{align*}
Here $W_\mu$ is the irreducible $L$-representation of 
highest weight $\mu$, and 
$m(k,\mu)=\dim_\C \mbox{Hom}_L(W_\mu, H^0(X, \mathscr{L}^k))$
its multiplicity in $H^0(X, \mathscr{L}^k)$.

To find the decomposition above is in general a very hard problem and 
only a few special cases are known. Instead, it has turned out to be 
fruitful to consider decompositions asymptotically in a certain sense. 
Indeed, if $m(1, \mu) \neq 0$, then $m(k, k\mu) \neq 0$
for all $k \in \N$. Moreover, the multiplicity  $m(k, k\mu)$
is a polynomial function of $k$ (cf. \cite{Ok96}), where 
the leading coefficient is given by the Euclidean volume of a certain 
compact convex set, namely of a slice of an Okounkov body 
for a line bundle defined by intersecting the Okounkov body 
with a certain affine subspace. However, this body need not 
in general be polyhedral. In the sequel we will see that 
the leading coefficient of $m(k, k\mu)$ can be interpreted
as the volume of a polyhedral Okounkov body for an ample line 
bundle, depending on $\mu$, and, hence, also as the self-intersection 
number of this line bundle.

By the Borel-Weil theorem, the highest weight representation 
$W_\mu$ 
can be realized geometrically as the space of 
sections of a line bundle as
\begin{align*}
W_\mu \simeq H^0(Y_\mu, \mathscr{L}_\mu),
\end{align*}
where $Y_\mu=L/Q_\mu$ is a flag variety for $L$ and 
$\mathscr{L}_\mu \rightarrow Y_\mu$ is an ample $L$-equivariant 
line bundle.
The multiplicities $m(k, k\mu)$ can now also 
be given a geometric interpretation. For this, let 
$Y⁻_\mu$ denote the variety having the same  point-set as $Y_\mu$, but
equipped with the opposite complex structure, i.e., the structure sheaf 
of $Y^-_\mu$ is given by the complex-conjugate of the structure 
sheaf of $Y_\mu$. The variety $Y⁻_\mu$ can also be described 
as the flag variety $Y^-_\mu=L/Q^-_\mu$, where $Q^-_\mu$ is 
the opposite parabolic subgroup of $Q_\mu$. Moreover, let
$T \subseteq Q_\mu \cap Q^-_\mu$ be a maximal torus such that 
the line bundle $\mathscr{L}_\mu$ is defined as 
\begin{align*}
\mathscr{L}_\mu=L \times_{\chi(\mu)} \C
\end{align*}
for a character $\chi(\mu): T \rightarrow \C^\times$ extended uniquely 
to a character of $Q_\mu$ (which we also denote by $\chi(\mu)$). 
Now we can defined the line bundle 
$$\mathscr{L}^-_\mu:=L \times_{\chi^-(\mu)} \C $$ 
by extending the character $\chi(\mu): T \rightarrow \C^\times$
to a character $\chi⁻(\mu): Q^-_\mu \rightarrow \C^\times$ of $Q^-_\mu$.
The multiplicity $m(k, k\mu)$ is then given by
\begin{align*}
m(k, k\mu)=\mbox{dim} \, H^0(X \times Y^-_\mu, 
\mathscr{L}^k \boxtimes ((\mathscr{L}^-_\mu)^k)^*)^L,
\end{align*}
where $H^0(X \times Y^-_\mu, 
\mathscr{L}^k \boxtimes ((\mathscr{L}^-_\mu)^k)^*)^L$ denotes the 
space of $L$-invariant sections of the exterior tensor product 
bundle $\mathscr{L}^k \boxtimes ((\mathscr{L}^-_\mu)^k)^*$
over $X \times Y^-_\mu$.

With the aim of putting the multiplicity formula above into the framework 
of spaces of global sections, 
the multiplicity can also be described as the dimension of the 
space of all sections of a coherent sheaf;
\begin{align*}
m(k, k\mu)=\mbox{dim} H^0(Z(\mu), \mathcal{L}(k, \mu)).
\end{align*}
Here $Z(\mu):=X \times Y^-_\mu//L$ is the 
GIT-quotient of $X \times Y^-_\mu$ with respect to 
the line bundle $\mathscr{L} \boxtimes (\mathscr{L}^-_\mu)^*$
and the action of $L$. For a thorough treatment of this result (in the case 
when $X$ is nonsingular), we 
refer to \cite{S95}, where the space $Z(\mu)$ is 
also proven to be homeomorphic to the symplectic quotient coming 
from the natural moment map 
$$ \mu: X \times Y^-_\mu \rightarrow \mathfrak{k}^*,$$
where $\mathfrak{k}$ is the Lie algebra of a maximal compact 
subgroup $K$ of $L$ which acts on $X \times Y^-_\mu$ in 
a Hamiltonian fashion. 
The sheaf $\mathcal{L}(k, k\mu)$ is naturally 
given as the $L$-invariants of the direct image 
sheaf $\pi_*(\mathscr{L}^k \boxtimes ((\mathscr{L}^-_\mu)^k)^*)$. 
Here $\pi$ is the quotient map onto $Z(\mu)$ from 
the set of semi-stable points of $X \times Y^-_\mu$. 
The sheaf $\mathcal{L}(k, k\mu)$ is a coherent sheaf, although 
not in general a line bundle. However, there exists 
a natural number $q$, such that for all $k \in \N$,
the $kq$-th tensor power 
$(\mathscr{L} \boxtimes (\mathscr{L}^-_\mu)^*)^{kq}$
descends to a line bundle $\mathcal{L}(kq, kq\mu)$
by the above construcion of taking $L$-invariants of 
the direct image sheaf under $\pi$. 
Moreover, $q$ can even be chosen so that the line bundle 
$\mathcal{L}(q, q\mu)$ is ample 
(cf. \cite[Thm 2.17]{S95}).

Even though $X \times Y^-_\mu$ is a nonsingular variety when $X$ is 
nonsingular, the quotient $Z(\mu)$ is in general singular. However, it 
has rational singularities, so it is in particular a normal 
variety (cf. \cite{S95}).

We can now apply the result from the previous section to 
this setting where $X=Z(\mu)$ and 
$\mathscr{L}=\mathcal{L}(q, q\mu)$. We then 
have the following result.

\begin{thm}
Let 
$$Z_0  \subseteq Z_1 \subseteq \cdots \subseteq Z_n=Z(\lambda, \mu),$$
where $n=\mbox{dim} Z(\mu)$ be a flag of irreducible normal 
subvarieties defined by a generic tuple $(\xi,\ldots, \xi_{n-1})$
of sections $\xi_1,\ldots, \xi_n \in H^0(Z(\mu), \mathcal{L}(q, q\mu))$
and a point $Z_0 \in Z_1$ and let 
$\Delta=\mbox{conv}\{0,be_1,e_2,\ldots, e_n\}$
be the associated Okounkov body for the line bundle $\mathcal{L}(q, q\mu)$. 
Then
\begin{align*}
\lim_{k \rightarrow \infty}\frac{m(kq,kq\mu)}{k^n}=\mbox{Vol}(\Delta)=b.
\end{align*}
Moreover, the volume $\mbox{Vol}(\Delta)=b$ is also given by the 
self-intersection number of the ample divisor $\mathcal{L}(q, q\mu)$.
\end{thm}

A particularly interesting special case of the above setting 
is the one when $X=G/P$ is a flag variety of a complex
reductive group $G$ containing $L$ as a subgroup, and 
$\mathscr{L}=\mathscr{L}_\lambda$ is a homogeneous 
line bundle, so that the space of 
sections $H^0(X, \mathscr{L}_\lambda^k)$ is 
the Borel-Weil realization of the irreducible 
$G$-module $V_{k\lambda}$ of highest weight $k\lambda$.
The decomposition under $L$ then amounts to the 
branching law
\begin{align*}
V_{k\lambda}=\sum_\mu m(k\lambda, \mu) W_\mu.
\end{align*}

\begin{rem}
The part of the main theorem which describes the 
asymptotics of the multiplicity function $m(kq, k\mu)$ as a 
self-intersection number can also be derived from the Riemann-Roch theorem 
for singular varieties, given vanishing results for the higher cohomology 
groups of the sheaf $\mathcal{L}(q,q\mu)$. Moreover, irrespective of 
the shape of the Okounkov body on the GIT-quotient, the 
asymptotic multiplicity would be given as the self intersection 
number of a divisor.
However, the polyhedral nature of the Okounkov body $\Delta$ above 
also allows us, for fixed $k \in \N$, to approximate the 
multiplicity $m(kq,kq\mu)$l by counting the number of integral points 
of the convex polytope $k\Delta$. For a precise formulation of the 
meaning of this approximation, we refer to \cite[Cor. 1.5]{KK09}, 
which is concerned with comparing the points of a semigroup with that of its 
saturation with respect to the closed convex cone generated by the semigroup.
In this sense, our result fits into a general philosophy of 
seeking \emph{polyhedral expressions} for multiplicities
(cf. \cite{Lit}, \cite{BZ}).
\end{rem}

\section{Concluding remarks}
The subvarieties $X_i$ occurring in the flag \eqref{E: flag} are unfortunately only 
defined very implicitly. It would be interesting to find more explicit examples of 
admissible flags yielding polyhedral Okounkov bodies. In particular, in the original 
setting developed by Okounkov, involving the presence of a group action, it is desirable 
to have a flag which is invariant under a Borel subgroup in order to relate the 
Okounkov body to asymptotics of multiplicities for irreducible representations.
In particular, this is crucial for being able to work with a fixed Okounkov body, 
for which certain slices describe multiplicities for subrepresentations, 
as opposed to using a variety and a line bundle which depend on 
$\mu$ for constructing a $\mu$-dependent Okounkov body in order to treat 
the multiplicities $m(k, k\mu)$. However, the condition of having 
an invariant flag is not well compatible with our approach using a 
flag where each $X_k$ is a complete intersection of divisors. Indeed, one  
can not hope for a Bertini-type theorem which provides invariant flags. 
For instance, in the case of a flag variety $X=G/P$, the 
flags of Borel-invariant subvarieties of $X$ consist of Schubert varieties, 
and an inclusion $X_k \subseteq X_{k+1}$ of Schubert varieties 
is not in general given by a Cartier divisor.



\begin{thebibliography}{999999}

\bibitem[AKL12]{akl}
Anderson, D., K\"uronya, A., Lozovanu, V., 
{\it Okounkov bodies of finitely generated divisors}, 
International Mathematics Research Notices
2013; doi: 10.1093/imrn/rns286

\bibitem[B97]{B}
Bauer, T.,
{\it Seshadri constants of quartic surfaces},
{Math. Ann.} {\bf 309} (1997), 3, 475--481

\bibitem[BZ01]{BZ}{Berenstein, A., Zelevinsky, A.},
{\it Tensor product multiplicities, canonical bases and totally
              positive varieties}, 
              Invent. Math. {\bf 143} (2001), 1, 77--128

\bibitem[B86]{Br86}
Brion, M.,
{\it Sur l'image de l'application moment},
{S\'eminaire d'alg\`ebre {P}aul {D}ubreil et {M}arie-{P}aule
              {M}alliavin ({P}aris, 1986)},
{Lecture Notes in Math.},
{ \bf 1296}, {177--192},
{Springer},
{Berlin},
{1987}



\bibitem[J10]{jow}
{Jow, Shin-Yao},
{\it Okounkov bodies and restricted volumes along very general
              curves},
Adv. Math. {\bf 223}, (2010), 4, 1356--1371


\bibitem[K11]{Kav11}
{Kaveh, K.},
{\it Crystal bases and Newton-Okounkov bodies},
preprint, arxiv.org/abs/1101.1687

\bibitem[KK09]{KK09}
{Kaveh, K., Khovanskii, A.G.},
{\it Newton convex bodies, semigroups of integral points, graded algebras 
and intersection theory}, 
preprint, arxiv.org/abs/0904.3350

\bibitem[KK10]{KK10}
{Kaveh, K., Khovanskii, A.G.},
{\it Convex bodies associated to actions of reductive groups},
preprint, arxiv.org/abs/1001.4830

\bibitem[L04]{laz}
{Lazarsfeld, R.,}
``Positivity in Algebraic Geometry I", Springer, 2004

\bibitem[LM09]{LM09}
{Lazarsfeld, R., Musta{\c{t}}{\u{a}}, M.},
{\it Convex bodies associated to linear series},
{Ann. Sci. \'Ec. Norm. Sup\'er.} {\bf 42}
(2009), 783--835

\bibitem[L98]{Lit}
{Littelmann, P.},
{\it Cones, crystals, and patterns}, Transform. Groups 
{\bf 3} (1998), 2, 145--179

\bibitem[O96]{Ok96}
{Okounkov, A.},
{\it Brunn-{M}inkowski inequality for multiplicities},
{Invent. Math.} {\bf 125} (1996), 405--411

\bibitem[S50]{S50}
{Seidenberg, A.},
{\it The hyperplane sections of normal varieties},
{Trans. Amer. Math. Soc.} {\bf 69} (1950), 357--386

\bibitem[SS14]{SS14}
Schmitz, D., Sepp\"anen, H.,
{\it On the polyhedrality of global Okounkov bodies}, 
in preparation

\bibitem[S95]{S95}
Sjamaar, R.,
{\it Holomorphic slices, symplectic reduction and multiplicities of
              representations},
Ann. of Math. (2) {\bf 141} (1995), no. 1,
87--129

\end{thebibliography}
\end{document}